\documentclass[12pt,reqno,twoside]{amsart}

\usepackage{amssymb,amsmath,amscd,enumerate,verbatim}
\usepackage[colorlinks=true,linkcolor=blue,citecolor=blue]{hyperref}
\usepackage[latin1]{inputenc}
\usepackage{color}
\usepackage{graphicx}

\usepackage[all]{xy}
\usepackage{pstricks, pst-3d}

\newcommand{\mm}{\mathfrak m}
\newcommand{\nn}{\mathfrak n}

\newcommand{\Z}{\mathbb{Z}}

\newcommand{\N}{\mathbb{N}}
\newcommand{\Q}{\mathbb{Q}}

\newcommand{\Fc}{\mathcal{F}}

\DeclareMathOperator{\pnt}{\raise 0.5mm \hbox{\large\bf.}}

\DeclareMathOperator{\depth}{depth}

\DeclareMathOperator{\Img}{Im}

\DeclareMathOperator{\gr}{gr}

\DeclareMathOperator{\Tor}{Tor}

\DeclareMathOperator{\Ass}{Ass}

\DeclareMathOperator{\lind}{ld}
\DeclareMathOperator{\glind}{gl\,ld}
\DeclareMathOperator{\linp}{lin}

\DeclareMathOperator{\Rees}{Rees}

\DeclareMathOperator{\pdeg}{pdeg}
\DeclareMathOperator{\projdim}{pd}

\def\+#1{\relax\ifmmode\if\noexpand #1\relax \mathop{\kern
    0pt^+{#1}}\nolimits\else \kern 0pt^+\!#1 \fi\else$^*$#1\fi}

\newcommand{\vphi}{\varphi}

\let\phi=\varphi
\let\:=\colon

\newtheorem{thm}{\bf Theorem}[section]
\newtheorem{lem}[thm]{\bf Lemma}
\newtheorem{cor}[thm]{\bf Corollary}

\theoremstyle{definition}
\newtheorem{defn}[thm]{\bf Definition}

\theoremstyle{plain}
\newtheorem*{thm*}{Theorem}
\newtheorem*{lem*}{Lemma}
\newtheorem*{cor*}{Corollary}
\newtheorem*{claim*}{Claim}
\newtheorem*{defn*}{Definition}

\theoremstyle{remark}
\newtheorem{rem}[thm]{Remark}

\newtheorem{ex}[thm]{Example}

\numberwithin{equation}{section}
\title{On the asymptotic behavior of the linearity defect}

\author{Hop D. Nguyen}
\address{Fachbereich Mathematik/Informatik, Institut f\"ur Mathematik, Universit\"at Osnabr\"uck, Albrechtstr. 28a, 49069 Osnabr\"uck, Germany}
\address{Dipartimento di Matematica, Universit\`a di Genova, Via Dodecaneso 35, 16146 Genoa, Italy}
\email{ngdhop@gmail.com}
\author{Thanh Vu}
\address{Department of Mathematics, University of Nebraska-Lincoln, Lincoln, NE 68588, USA}
\email{tvu@unl.edu}
\thanks{This work is partially supported by the NSF grant DMS-1103176.}


\begin{document}

\begin{abstract}
This work concerns the linearity defect of a module $M$ over a noetherian local ring $R$, introduced by Herzog and Iyengar in 2005, and denoted by $\lind_R M$. Roughly speaking, $\lind_R M$ is the homological degree beyond which the minimal free resolution of $M$ is linear. In the paper, it is proved that for any ideal $I$ in a regular local ring $R$ and for any finitely generated $R$-module $M$,
each of the sequences $(\lind_R (I^nM))_n$ and $(\lind_R (M/I^nM))_n$ is eventually constant. The first statement follows from a more general result about the eventual constancy of the sequence $(\lind_R C_n)_n$ where $C$ is a finitely generated graded module over a standard graded algebra over $R$. The second statement follows from the first together with a result of Avramov on small homomorphisms.
\end{abstract}

\maketitle

\section{Introduction}
\label{sect_intro}
Let $(R,\mm)$ be a commutative noetherian local ring with maximal ideal $\mm$ and residue field $k=R/\mm$. Let $I\subseteq \mm$ be an ideal of $R$. The study of the asymptotic behavior of, for instances, graded components of a graded module, or powers of an ideal, is a classical topic in commutative algebra. A primal example of good asymptotic behavior is that if $I$ is $\mm$-primary, then the Hilbert-Samuel function of $I$, given by $n \mapsto \text{length of}~ (R/I^n)$, is eventually a polynomial function when $n$ tends to infinity. To give another example, Brodmann \cite{Br} proved that for any $I$, the sequence $(\depth (R/I^n))_n$ is eventually constant for $n$ large enough. Aside from qualitative statements, for each invariant, it is also of interest to study the point when nice asymptotic behaviors of powers occur and the nature of the limiting values at high enough powers. See, for example, \cite{Ch}, \cite{HH2} and the references therein for more details.

In this paper, we study the linearity defect introduced by Herzog and Iyengar \cite{HIy} (see Section \ref{sect_linearitydefect}). One of the motivations for studying the linearity defect is the research on the linear part of minimal free resolutions over the exterior algebras in \cite{EFS}. Similarly to the projective dimension, the finiteness of the linearity defect has strong consequences on the structure of a module: If $\lind_R M$ is finite, then the Poincar\'e series of $M$ is rational with denominator depending only on $R$ (see \cite[Proposition 1.8]{HIy}). However, there remain many open questions on the finiteness of the linearity defect; see \cite{AR}, \cite{CINR}, \cite{Se} for more details.

The linearity defect was studied by many authors, see for example \cite{AR}, \cite{CINR}, \cite{HIy}, \cite{IyR}, \cite{OkaYan}, \cite{Se}, \cite{Y}. Nevertheless, it is still an elusive invariant. The problem is highly non-trivial as to bound efficiently the linearity defect even for familiar classes of ideals like monomial ideals. Beyond componentwise linear ideals \cite{HH} (which have linearity
defect zero), there are few interesting and large enough classes of ideals whose linearity defect is known. In this paper, the above remarks notwithstanding, we show that the linearity defect behaves in a pleasant way asymptotically. Denote by $\glind R=\sup\{\lind_R M: ~\textnormal{$M$ is a finitely generated $R$-module}\}$ the so-called {\em global linearity defect} of $R$. For instances, regular rings (or more generally local rings which are both Koszul and Golod, see \cite[Corollary 6.2]{HIy}) have finite global linearity defect. The main result  of this paper is
\begin{thm}
\label{thm_main1}
Let $(R,\mm)$ be a local ring such that $\glind R<\infty$, e.g., $R$ is regular. Let $I\subseteq \mm$ be an ideal of $R$ and $M$ a finitely generated $R$-module. Then for all $n$ large enough, each of the numbers $\lind_R (I^nM), \lind_R (I^nM/I^{n+1}M)$, and $\lind_R (M/I^nM)$ is a constant depending only on $I$ and $M$.
\end{thm}
The proof essentially depends on a characterization of linearity defect in terms of Tor due to \c{S}ega \cite{Se}, and the theory of Rees algebras. Theorem \ref{thm_main1} can be divided into two parts: the first concerns with the sequences $(\lind_R (I^nM))_n, (\lind_R (I^nM/I^{n+1}M))_n$ and the second with $(\lind_R (M/I^nM))_n$. The proof of the second part is obtained from the proof of the first and a result of Avramov \cite{Avr1} on small homomorphisms. The first part is obtained from the following general result. Below, recall that $S$ is called a standard graded algebra over $R$ if it is an $\N$-graded ring with $S_0=R$ and $S$ is generated over $R$ by elements of $S_1$.
\begin{thm}
\label{thm_largecomponents}
Let $(R,\mm)$ be a local ring such that $\glind R<\infty$. Let $S$ be a standard graded algebra over $R$ and $C$ a finitely generated graded $S$-module. Then the sequence $(\lind_R C_n)_n$ is eventually constant.
\end{thm}

Theorem \ref{thm_largecomponents} is motivated by previous work of Herzog and Hibi \cite[Theorem 1.1]{HH2} on depth. In Theorem \ref{thm_largecomponents}, if we allow $S$ to be generated by finitely many elements of positive degrees, then $(\lind_R C_n)_n$ is quasiperiodic (namely, there exist an integer $p\ge 1$ and numbers $\ell_0,\ldots,\ell_{p-1}$ such that for all $n$ large enough, $\lind_R C_n=\ell_i$ if $n$ is congruent to $i$ modulo $p$, where $0\le i\le p-1$). In the end of Section \ref{sect_largepowers}, we turn to the asymptotic linearity defect of the integral closure $\overline{I^n}$ and the saturation $\widetilde{I^n}$ of $I^n$. Theorem \ref{thm_largecomponents} is also suitable for studying the sequence $(\lind_R \overline{I^n})_n$, at least over regular local rings. We give an example in which the sequence $(\lind_R \widetilde{I^n})_n$ is quasiperiodic but not eventually constant. We do not know any example for which $(\lind_R \widetilde{I^n})_n$ is non-quasiperiodic.

Theorem \ref{thm_main1} implies that the sequence $(\lind_R I^n)_n$ is eventually constant. It remains mysterious to us how to bound effectively the asymptotic value of the sequence $(\lind_R I^n)_n$. A rare result in this direction is \cite[Theorem 2.4]{HHO}. There the authors establish a necessary and sufficient condition for all the powers of a polynomial ideal to have linearity defec zero, using the theory of $d$-sequences \cite{Hu}. It would be interesting to study possible generalizations and analogues of this result.

We assume that the reader is versed with the standard concepts and facts of commutative algebra, which may be found in \cite{BH}, \cite{Eis}. 
\section{Linearity defect}
\label{sect_linearitydefect}
Let $(R,\mm)$ be a commutative noetherian local ring with residue field $k$. Let $I$ be a proper ideal of $R$. Let $M$ be a finitely generated $R$-module. We call $\displaystyle \gr_I M=\bigoplus_{j\ge 0} \frac{I^jM}{I^{j+1}M}$ the associated graded module of $M$ with respect to $I$.

Let $F$ be the minimal free resolution of $M$:
\[
\xymatrixcolsep{5mm}
\xymatrixrowsep{2mm}
\xymatrix{
F:	\cdots \ar[rr] &&  F_i \ar[rr]&&  F_{i-1}  \ar[rr]&&  \cdots \ar[rr] && F_1\ar[rr] && F_0 \ar[rr]&& 0.
}
\]
Since $\partial(F)\subseteq \mm F$ the (by homological degree) graded submodule
\[
\xymatrixcolsep{5mm}
\xymatrixrowsep{2mm}
\xymatrix{
\Fc^iF: 	\cdots \ar[rr] && F_{i+1} \ar[rr] &&  F_i \ar[rr]&&  \mm F_{i-1}  \ar[rr]&&  \cdots \ar[rr] && \mm^{i-j} F_j  \ar[rr]&&  \cdots
}
\]
of $F$ is stable under the differential; said otherwise, $\Fc^iF$ is a subcomplex of $F$.

We define the so-called {\em linear part} of $F$ by the formula
$$
\linp^R F:=\bigoplus_{i=0}^{\infty} \frac{\Fc^iF}{\Fc^{i+1}F}.
$$
Note that $\linp^R F$ is a complex of graded modules over $\gr_{\mm}R$, and $(\linp^R F)_i=(\gr_{\mm}F_i)(-i)$ for every $i\ge 0$.  Following \cite{HIy}, the {\it linearity defect} of $M$ is defined by
\[
\lind_R M:=\sup\{i: H_i(\linp^R F)\neq 0\}.
\]
If $M\cong 0$, we set $\lind_R M=0$. This convention guarantees that the maximal ideal $(0)$ of the field $k$ has linearity defect zero.
\begin{rem}
(1) The definition of the linear part works also for standard graded algebras. Let $(R,\mm)$ be a standard graded algebra over a field $k$, with the graded maximal ideal $\mm$, and $M$ a finitely generated graded $R$-module. Let $F$ be the minimal graded free resolution of $M$ over $R$. Then the above construction of the linear part goes through for $F$.

In this case, the linear part $\linp^R F$ has a simple interpretation. For each $i\ge 1$, apply the following rule to all entries in the matrix representing the map $F_i\to F_{i-1}$: keep it if it is a linear form, and replace it by $0$ otherwise. Then the resulting complex is $\linp^R F$.

(2) Although our results deal only with local rings, their counterparts for standard graded algebras are also valid and can be established by the same methods.
\end{rem}
\section{Asymptotic behavior of the linearity defect}
\label{sect_largepowers}
Let $(R,\mm)$ be a noetherian local ring such that $\glind R<\infty$. We divide Theorem \ref{thm_main1} into two parts: the first concerns with the stability of $\lind_R (I^nM)$ and $\lind_R (I^nM/I^{n+1}M)$ for large $n$, and the second with the stability of $\lind_R (M/I^nM)$. The proof of the second part follows from the first part and a result due to Avramov on small homomorphisms. The first part follows from Theorem \ref{thm_largecomponents}.
\subsection{The first part of Theorem \ref{thm_main1}}
Before deducing the first part of Theorem \ref{thm_main1} from Theorem \ref{thm_largecomponents}, recall that $\Rees(I)=R\oplus I \oplus I^2 \oplus \cdots$ denote the Rees algebra of $I$, whose grading is given by $\deg I^n=n$. Since $R$ is noetherian, $\Rees(I)$ is a standard graded $R$-algebra.

\begin{proof}[Proof of the first part of Theorem \ref{thm_main1}]
Clearly $\bigoplus_{n\ge 0} I^nM$ and $\bigoplus_{n\ge 0} I^nM/I^{n+1}M$ are finitely generated graded modules over $\Rees(I)$. By Theorem \ref{thm_largecomponents}, we see that each of the sequences $(\lind_R I^nM)_n$ and $(\lind_R I^nM/I^{n+1}M)_n$ is eventually constant.
\end{proof}
To prove Theorem \ref{thm_largecomponents}, it is harmless to replace $S$ by a standard graded polynomial ring over $R$ which surjects onto it. Hence below, we will adopt the hypothesis that $S=R[y_1,\ldots,y_m]$ is a polynomial extension of $R$ with $\deg y_i=1$. With this hypothesis, we make Theorem \ref{thm_largecomponents} more precise by giving an upper bound for the smallest integer from which the sequence $(\lind_R C_n)_n$ becomes constant. This requires certain information about the minimal graded free resolution of $C$ as an $S$-module. 
\begin{defn}
\label{defn_pfunction}	
For each finitely generated graded $S$-module $C$, let $\pdeg_S(C)$ be the minimal number such that $C_i=0$ for all $i\ge \pdeg_S(C)$ or $C_i\neq 0$ for all $i\ge \pdeg_S(C)$. If $C=0$, we set $\pdeg_S (0)=-\infty$. Note that $\pdeg_S(C)$ is well-defined since $S$ is standard graded.
\end{defn}
We can compute the number $\pdeg_S(C)$ effectively, using two simple facts:
\begin{enumerate}
\item $\pdeg_S(C)=\pdeg_{S/\mm S}(C/\mm C)$. This is by Nakayama's lemma.
\item $\pdeg_{S/\mm S}(C/\mm C)$ is bounded above by the point where the Hilbert function and Hilbert polynomial of $C/\mm C$ as a module over $S/\mm S=k[y_1,\ldots,y_m]$, start to agree. The latter number is given, for example, in \cite[Proposition 4.12]{BH}.
\end{enumerate}

Given a finitely generated graded $S$-module $C$, define the constant $N(C)$ as follows. For $i=0$, denote $n(0)=\pdeg_S(C)$.

For $1\le i\le \min\{\glind R,\projdim_S C\}$, denote by $c(i,q)$ the number 
$$
c(i,q):=\pdeg_S\left(\Img \mu^{i,q} \right),
$$
where $\mu^{i,q}$ denotes the map $\Tor^S_i(S/\mm^{q+1}S,C)\to \Tor^S_i(S/\mm^qS,C)$.
Let
\[
\xymatrixcolsep{5mm}
\xymatrixrowsep{2mm}
\xymatrix{
	\cdots \ar[rr] &&  F_i \ar[rr]&&  F_{i-1}  \ar[rr]&&  \cdots \ar[rr] && F_0 \ar[rr] && 0
}
\]
be the minimal graded free resolution of $C$ over $S$; a reference for the existence of such a resolution is \cite[Section 1.5]{BH}. 

Denote by $M_i=\Img(F_i\to F_{i-1})$ the $i$-syzygy of $C$. Using the Artin-Rees lemma, choose $T(i)$ to be the minimal number $h\ge 1$ such that 
$$\mm^{q}F_{i-1}\cap M_i=\mm^{q-h}(\mm^hF_{i-1}\cap M_i)$$ 
for all $q\ge h$. Denote
$$
n(i)=\max\{c(i,1),\ldots,c(i,T(i))\}.
$$
Finally, let $N(C)=\max\left\{n(0),n(1),\ldots,n(\min\{\glind R,\projdim_S C\})\right\}$. Since $\glind R$ is a finite number, $N(C)$ is also finite. 
\begin{rem}
In principal, the numbers $T(i)$ in the definition of $N(C)$ should not be difficult to compute. Indeed, let $\nn$ denote the graded maximal ideal of the associated graded ring 
$$
\gr_{\mm S}(S)=\left(\gr_{\mm}R\right)\otimes_R S.
$$ 
In concrete terms, $\nn=(\mm/\mm^2)\otimes_R S$. Furthermore, denote by $K_i$ the kernel of the natural surjective map $\gr_{\mm S}(F_{i-1}) \to \gr_{\mm S}(F_{i-1}/M_i)$. Then there is an equality
\[
T(i)=\sup\{q: (K_i/\nn K_i)_q \neq 0\}.
\]
The proof is straightforward; see \cite[Proposition 2.1]{HWY} for an analogous statement.
\end{rem}
As explained above, we can deduce Theorem \ref{thm_largecomponents} from the next result
\begin{thm}
\label{thm_main2}
Let $(R,\mm)$ be a local ring with $\glind R<\infty$. Let $S$ be a standard graded polynomial ring over $R$ and $C$ a finitely generated graded $S$-module. Then for all $n\ge N(C)$, $\lind_R C_n$ is a constant independent of $n$.
\end{thm}
In the proof, we will use the following characterization of linearity defect due to \c{S}ega.
\newpage
\begin{thm}[\c {S}ega, {\cite[Theorem 2.2]{Se}}]
\label{thm_Sega}
Let $(R,\mm)$ be a noetherian local ring, $M$ a finitely generated $R$-module, and $d\ge 0$ an integer. The following statements are equivalent:
\begin{enumerate}[\quad \rm(i)]
\item $\lind_R M\le d$;
\item The natural morphism $\Tor^R_i(R/\mm^{q+1}, M)\longrightarrow \Tor^R_i(R/\mm^q, M)$ is the zero map for every $i>d$ and every $q\ge 0$.
\end{enumerate}
\end{thm}
\begin{proof}[Proof of Theorem \ref{thm_main2}]
Since $S$ is a flat $R$-algebra, there is an isomorphism of $R$-modules
\begin{equation}
\label{eq_passtoS}
\Tor^R_i(R/\mm^q,C_n)\cong \Tor^S_i(S/\mm^qS, C)_n
\end{equation}
for all $i,q\ge 0, n\in \Z$.

Let $N=N(C)$ and $e=\sup \{\lind_R C_n: n\ge N\} \le \glind R<\infty$. We prove that $\lind_R C_n=e$ for any $n\ge N$. There is nothing to do if $e=0$, so we assume that $e\ge 1$. Note that, since $\Tor^S_i(S/\mm^qS,C)=0$ for $i>\projdim_S C$, Isomorphism \eqref{eq_passtoS} yields $e\le \min\{\glind R, \projdim_S C\}$.

Denote by $\mu^{e,q}_n$ the following map 
$$\Tor^S_e(S/\mm^{q+1}S,C)_n\longrightarrow \Tor^S_e(S/\mm^qS,C)_n.$$ Choose $m\ge N$ such that $\lind_R C_m=e$. Since $\lind_R C_m=e>e-1$, Theorem \ref{thm_Sega} implies that $\mu^{e,\overline{q}}_m\neq 0$ for some $\overline{q}\ge 0$. To prove the inequality $\lind_R C_n \ge e$, also by Theorem \ref{thm_Sega}, it suffices to show that $\mu^{e,q}_n\neq 0$ for some $q\ge 0$. 

Firstly, consider the case $\overline{q}< T(e)$. Since $n,m\ge N \ge c(e,\overline{q})$, the definition of $c(e,\overline{q})$ implies that $\mu^{e,\overline{q}}_n$ and $\mu^{e,\overline{q}}_m$ are both zero or both non-zero. This implies that $\mu^{e,\overline{q}}_n\neq 0$, as desired.

Secondly, consider the case $\overline{q}\ge T(e)$. Denote $T=T(e)$, we claim that $\mu^{e,T}_n\neq 0$. As $m,n\ge c(e,T)$, $\mu^{e,T}_m$ and $\mu^{e,T}_n$ are both zero or both non-zero, so it suffices to prove that $\mu^{e,T}_m\neq 0$. Assume the contrary, so $\mu^{e,T}_m=0$. Let $F$ be the minimal graded free resolution of $C$ over $S$. Denote $M_i=\Img(F_i\to F_{i-1})$, the $i$-th syzygy of $C$ as an $S$-module if $i\ge 1$ and $M_0=C$. Denote $M=M_e$ and $P=F_{e-1}$. Clearly 
\begin{align*}
\Tor^S_e(S/\mm^qS,C) &\cong \Tor^S_1(S/\mm^qS,M_{e-1}) \cong \frac{\mm^qP\cap M}{\mm^q M},\\
\Img \mu^{e,q}&\cong \frac{\mm^{q+1}P\cap M+\mm^q M}{\mm^q M}.
\end{align*}
 The equality $\mu^{e,T}_m=0$ then yields
\begin{equation}
\label{eq_m}
(\mm^{T+1}P\cap M+\mm^T M)_m=(\mm^T M)_m.
\end{equation}
We will show that $\mu^{e,\overline{q}}_m=0$. Indeed, 
\begin{align*}
\left(\mm^{\overline{q}+1}P\cap M +\mm^{\overline{q}}M\right)_m &=\left(\mm^{\overline{q}-T}(\mm^{T+1}P\cap M +\mm^TM)\right)_m\\
&= \mm^{\overline{q}-T}\left(\mm^{T+1}P\cap M +\mm^TM\right)_m\\
&= \mm^{\overline{q}-T}(\mm^TM)_m\\
&=(\mm^{\overline{q}}M)_m.
\end{align*}
In the above string, the first equality holds because of the inequality $\overline{q}\ge T=T(e)$ and the definition of $T(e)$, the second and fourth because $\mm\subseteq S_0$, the third because of \eqref{eq_m}.

Therefore, $\mu^{e,\overline{q}}_m=0$. But this is a contradiction, so the proof of the theorem is finished.
\end{proof}
\subsection{The second part of Theorem \ref{thm_main1}}
We state a simple result which is useful for the proof of the second part of Theorem \ref{thm_main1}, which concerns with the stability of the sequence $(\lind_R (M/I^nM))_n$.
\begin{lem}
\label{lem_mapofresol}
Let $0\longrightarrow M \longrightarrow P \longrightarrow N \longrightarrow 0$ be an exact sequence of non-trivial, finitely generated $R$-modules. Let $F, G$ be the minimal free resolution of $M, P$, respectively. Assume that there is a lifting $\vphi: F\longrightarrow G$ of $M\to P$ such that $\vphi(F)\subseteq \mm^2G$. Then there is an equality
\[
\lind_R N =\max\{\lind_R P, \lind_R M+1\}.
\]
\end{lem}
\begin{proof}
Let $W$ be the mapping cone of $\vphi$. Since $\vphi(F)\subseteq \mm^2G$, $W$ is a minimal free resolution of $N$. By simple computations, we get a direct sum decomposition
\[
\linp^R W\cong \linp^R G \oplus (\linp^RF)[-1]
\]
Hence by accounting, $\lind_R N=\max\{\lind_R P, \lind_R M+1\},$ as desired.
\end{proof}
Combining the first part of Theorem \ref{thm_main1} with a result due to Avramov \cite{Avr1}, we will provide the
\begin{proof}[Proof of the second part of Theorem \ref{thm_main1}]
We have to prove the eventual constancy of the sequence $(\lind_R (M/I^nM))_n$. If there exists an $n_0\ge 1$ such that $I^{n_0}M=0$ then $I^nM=0$ for all $n\ge n_0$, and we are done. Hence below, we assume that $I^nM\neq 0$ for all $n\ge 1$.

By \cite[Corollary A.4]{Avr1}, there exists $d\ge 1$ such that for any $P\subseteq \mm^dM$, the map $\Tor^R_i(k,P)\to \Tor^R_i(k,M)$ is zero for all $i\ge 0$. Applying the same result, there exists $e\ge 1$ such that for any $P\subseteq \mm^{d+e}M$, the map $\Tor^R_i(k,P) \to \Tor^R_i(k,\mm^dM)$ is zero for all $i\ge 0$.

Take $n\ge d+e$. Then the maps $\Tor^R_i(k, I^nM) \to \Tor^R_i(k,I^dM)$ and $\Tor^R_i(k,I^dM) \to \Tor^R_i(k,M)$ are zero for all $i\ge 0$. Let $F, G, H$ be the minimal free resolution of $I^nM, I^dM, M$, respectively. Take any lifting $\lambda: F\to G$ of the map $I^nM\to I^dM$, then $\lambda(F)\subseteq \mm G$. Similarly, for any lifting $\psi: G\to H$ of the map $I^dM\to M$, we have $\psi(G) \subseteq \mm H$. Therefore we obtain a lifting $\phi=\psi\circ \lambda: F\to H$ on the level of minimal free resolutions of the map $I^nM \to M$ which satisfies $\phi(F) \subseteq \mm^2 G$.

By Lemma \ref{lem_mapofresol}, we have for any $n\ge d+e$ the equality
\[
\lind_R (M/I^nM)=\max\{\lind_R M, \lind_R (I^n M)+1\}.
\]
By the first part of Theorem \ref{thm_main1}, for $n$ large enough, $\lind_R (I^nM)$ is a constant independent of $n$. Hence the same is true for $\lind_R (M/I^nM)$. This concludes the proof of the second part of Theorem \ref{thm_main1}.
\end{proof}
\subsection{Final remarks}
\label{subsect_finalrem}
The following example illustrates the value of the constant $N=N(C)$ in Theorem \ref{thm_main2}, with the help of Macaulay2 \cite{GS}. 
\begin{ex}
\label{ex_firststep}
Let $R=\Q[x,y,z]$ be a polynomial ring of dimension $3$ and $I=(x^2,xy,z^2)$. Denote $S=R[w_0,w_1,w_2]$ a standard graded polynomial extension of $R$ which surjects onto the Rees algebra $E=\Rees(I)$ by mapping $w_0\mapsto x^2, w_1 \mapsto xy, w_2 \mapsto z^2$. The ring $E$ has the following presentation
$$
E\cong \frac{S}{(w_0y-w_1x,w_0z^2-w_2x^2,w_1z^2-w_2xy)}.
$$
The minimal graded free resolution of $E$ over $S$ is as follows

\[
F: 0\rightarrow \begin{array}{c}
S(-2)\\
\bigoplus \\
S(-1)\\
\end{array}\xrightarrow{\left(
	\begin{array}{c c}
	w_2x & -z^2\\
         -w_1 & y\\   	
         w_0 & -x\\
        \end{array}\right)} S(-1)^3 \xrightarrow{\left(
		\begin{array}{c c c}
w_0y-w_1x & w_0z^2-w_2x^2 & w_1z^2-w_2xy\\
		\end{array}
				\right)} S \rightarrow 0.
\]

Using the notation of the proof of Theorem \ref{thm_main2}, we will show that $N=1$, namely all the powers of $I$ have the same linearity defect, which turns out to be $1$. Since $\projdim_S E=2<\glind R=3$, $N=\max\{n(0),n(1),n(2)\}$. The graded structure of $E$ tells us that $n(0)=\pdeg_S(E)=0$.

Let $J\subseteq S,M_2\subseteq G$ be the first and second syzygies of $E$, where $G$ denotes the module $F_1=S(-1)^3$. We claim that $T(1)=2$ and $T(2)=1$, namely,
\begin{align}
\label{eq_T_numbers1}
\mm^qS\cap J &=\mm^{q-2}(\mm^2S\cap J), ~\text{for all $q\ge 2$},\\
\mm^qG\cap M_2&=\mm^{q-1}(\mm G\cap M_2),~\text{for all $q\ge 1$}. \label{eq_T_numbers2}
\end{align}

For \eqref{eq_T_numbers1}: one sees immediately that both sides are equal to $\mm^{q-1}(w_0y-w_1x)S+\mm^{q-2}(w_0z^2-w_2x^2, w_1z^2-w_2xy)S$.

For \eqref{eq_T_numbers2}: we have $M_2=(w_2xe_1-w_1e_2+w_0e_3,-z^2e_1+ye_2-xe_3)$, where $e_1,e_2,e_3$ is the standard basis of $G$ sitting in degree $1$. It is not hard to check that both sides of \eqref{eq_T_numbers2} are equal to
\[
\mm^q(w_2xe_1-w_1e_2+w_0e_3)+\mm^{q-1}(-z^2e_1+ye_2-xe_3).
\]

The above arguments yield $n(1)=\max\{c(1,1),c(1,2)\}$ and $n(2)=c(2,1)$. We prove that $n(1)=1$ and $n(2)=-\infty$.

For each $q\ge 1$, $\Tor^S_1(S/\mm^qS,E)=\Tor^S_1(S/\mm^qS,S/J)=(J\cap \mm^qS)/(J\mm^qS)$. Therefore the image of $\Tor^S_1(S/\mm^{q+1}S,E)\to \Tor^S_1(S/\mm^qS,E)$ is $\Img \mu^{1,q}=(J\cap \mm^{q+1}S+J\mm^qS)/(J\mm^qS)$. Computations show that
\[
\Img \mu^{1,1}=\frac{S^2}{\mm S^2+(w_0e^1_1-w_1e^1_2)},
\]
where $e^1_1,e^1_2$ is a basis for $S^2$, both of degree $1$, and
\[
\Img \mu^{1,2}=\frac{S^5}{\mm S(e^2_3,e^2_4,e^2_5)+\mm^2S(e^2_1,e^2_2)+(-xe^2_1+ye^2_2,w_0e^2_1-w_1e^2_2+w_2e^2_5)},
\]
where $e^2_1,\ldots,e^2_5$ are a basis for $S^5$, all of them of degree $1$. Thanks to routine Gr\"obner basis arguments, the residue classes $\overline{w_0^ie^1_2} \in \Img \mu^{1,1}$ and $\overline{w_0^ie^2_2}\in \Img \mu^{1,2}$ are always non-zero for every $i\ge 0$. Hence $c(1,1)=c(1,2)=1$, and thus $n(1)=1$.

Denote by $f_1,f_2$ the standard basis of $F_2$ where $\deg f_1=2,\deg f_2=1$. Since $\Tor^S_2(S/\mm^S,E)=H_2(F\otimes_S S/\mm^2S)$, computations show that
\begin{enumerate}
\item $\Tor^S_2(S/\mm^S,E)$ is generated by $\overline{xf_2}, \overline{yf_2}, \overline{zf_2} \in F_2\otimes (S/\mm^2S)$,
\item $\Tor^S_2(S/\mm S, E)$ is generated by $\overline{f_2} \in F_2\otimes (S/\mm S)$.
\end{enumerate} 
As $\Tor^S_2(S/\mm S, E)$ is killed by $\mm S$, the map $\Tor^S_2(S/\mm^S,E) \to \Tor^S_2(S/\mm S, E)$ is zero; this yields $n(2)=c(2,1)=-\infty$.

Putting everything together, $N=\max\{n(0),n(1),n(2)\}=\max\{0,1,-\infty\}=1$.
\end{ex}
Theorem \ref{thm_main2} has the following consequence on the linearity defect of the integral closure of powers. Recall that the integral closure $\overline{I}$ of $I$ is the ideal consisting of elements $x\in R$ which satisfies a relation
\[
x^n+a_1x^{n-1}+\cdots+a_{n-1}x+a_n=0
\]
where $n\ge 1$ and $a_i\in I^i$ for any $1\le i\le n$.
\begin{cor}
\label{cor_integralcl}
Let $(R,\mm)$ be a regular local ring and $I\subseteq \mm$ an ideal. Then the sequence $(\lind_R \overline{I^n})_n$ is eventually constant.
\end{cor}
\begin{proof}
Denote $C=R\oplus \overline{I} \oplus \overline{I^2} \oplus \cdots$, then $C$ is a graded module over $\Rees(I)$ with $\deg \overline{I^n}=n$. By \cite[Proposition 5.3.4]{HS}, $C$ is a finitely generated $\Rees(I)$-module. An application of Theorem \ref{thm_largecomponents} yields the desired conclusion.
\end{proof}

Recall that the saturation of $I$ is $\widetilde{I}=\{x\in R: x\mm^d \subseteq I ~\text{for some $d\ge 1$}\}$. The next example shows that the (graded) analog of Corollary \ref{cor_integralcl} for saturation of powers does not hold. 
\begin{ex}
\label{ex_periodic}
Consider the ideal $I=(x(y^3-z^3),y(x^3-z^3),z(x^3-y^3)) \subseteq R=\mathbb C[x,y,z]$. The ideal $I$ defines a reduced set of $12$ points in $\mathbb P^2$, the so-called {\it Fermat configuration} (see the proof of \cite[Proposition 2.1]{HaS}).  We show that the saturation ideals $\widetilde{I^s}$ do not have eventually constant linearity defect.
	
For $s\ge 1$, denote by $I^{(s)}=R\cap \bigcap_{P\in \Ass(I)}I^sR_P$ the $s$-th symbolic power of $I$. Since $I$ is the defining ideal of a reduced set of points, we get that $\widetilde{I^s}=I^{(s)}$. From \cite[Proposition 1.1]{HaS}, we deduce that $\widetilde{I^{3s}}=(\widetilde{I^3})^s$ . By \cite[Theorem 2.4]{HHO}, $\lind_R \widetilde{I^{3s}}=0$ for all $s\ge 1$. Indeed, computations with Macaulay2 \cite{GS} show that $x,y+z,z$ is a $d$-sequence with respect to $\Rees(\widetilde{I^3})$. 
	
Now we show that $\lind_R \widetilde{I^{3s+1}}=1$ for all $s\ge 1$. First, since $\depth R/\widetilde{I^{3s+1}}\ge 1$, by \cite[Proposition 6.3]{CINR},  
$$
\lind_R R/\widetilde{I^{3s+1}}\le \dim R-1=2.
$$
Hence $\lind_R \widetilde{I^{3s+1}} \le 1$.

Let $H=(x^3-y^3)(y^3-z^3)(z^3-x^3)$. We will show that the minimal non-zero component of $\widetilde{I^{3s+1}}=I^{(3s+1)}$ is of degree $9s+4$ and 
$$
I^{(3s+1)}_{\left<9s+4\right>}=(H^s)I_{\left<4\right>}\cong I(-9s).
$$
If this is the case, then $\widetilde{I^{3s+1}}_{\left<9s+4\right>}$ has linearity defect at least 1, as $I$ does. (For the inequality $\lind_R I\ge 1$, use R\"omer's theorem \cite[Theorem 5.6]{IyR} and the fact that $I$ is generated in degree $4$ but does not have $4$-linear resolution). Hence $\lind_R \widetilde{I^{3s+1}}\ge 1$ for every $s\ge 1$. All in all, we obtain $\lind_R \widetilde{I^{3s+1}}=1$ for every $s\ge 1$.

Now for our purpose, it suffices to prove the following claim:
\begin{equation}
\label{eq_saturation_equality}
I^{(3s+1)}_{\left<d\right>}=(H^s)I_{\left<d-9s\right>} 
\end{equation}
holds for all $d\le 9s+4$. We are grateful to Alexandra Seceleanu for providing us the following nice argument.

We will proceed by induction on $s$; the starting case $s$ = 0 is trivial. Assume that $s >0$. 

Let $G$ be a homogeneous element of $I^{(3s+1)}$ of degree $d$. Here the geometry of the Fermat configuration comes into play. We have a decomposition $H=\prod_{i=1}^9 h_i$, where each $h_i$ is a linear form and no two of them are proportional. According to \cite[Section 1.1]{HaS}, for each $i$, $h_i$ passes through exactly $4$ points (among the $12$ points of the configuration). Moreover, each point of the configuration lies on $3$ of the $9$ lines defined by the $h_i$s. 

Now as $G$ lies in $I^{(3s+1)}$, $G$ passes through each point of the configuration with multiplicity at least $3s+1$. Thus the curves $(G)$ and $(h_i)$ intersect with multiplicity at least $4(3s+1)$, which is strictly larger than $d=(\deg G)\cdot(\deg h_i)$. From that, Bezout's theorem forces $G$ to be divisible by $h_i$ for all $1\le i\le 9$. In particular $G$ is divisible by $H$. Writing $G = H G'$, then as $H$ vanishes exactly $3$ times at each of the points, we must have $G' \in I^{(3s+1-3)}_{\left<d-9\right>}=I^{(3(s-1)+1)}_{\left<d-9\right>}$. Finally, the induction hypothesis gives us the claim.
	
So we conclude that the sequence $\lind_R \widetilde{I^{s}}$ is not eventually constant when $s$ goes to infinity.
\end{ex}
\begin{rem}
\label{rem_periodic}
More generally than Theorem \ref{thm_largecomponents}, one can prove the following: If $S$ is a noetherian $R$-algebra which is generated by elements of positive degrees, and $C$ is a finitely generated graded $S$-module, then the sequence $(\lind_R C_n)_n$ is quasiperiodic, i.e. there exist a number $p\ge 1$ and integral constants $\ell_0,\ldots,\ell_{p-1}$ such that for all $n\gg 0$, we have $\lind_R \widetilde{C_n}=\ell_i$ if $n$ is congruent to $i \in\{0,\ldots,p-1\}$ modulo $p$. 

The proof uses the fact that any high enough Veronese subring of $S$ is standard graded (after normalizing the grading), and Theorem \ref{thm_largecomponents}. We leave the details to the interested reader (see an analogous statement in \cite[Theorem 4.3]{CHT}).

By \cite[Theorem 4.3]{NS}, for the ideal $I$ in Example \ref{ex_periodic}, the graded $R$-algebra $R\oplus \widetilde{I} \oplus \widetilde{I^2}\oplus \cdots$ is finitely generated. This fact and the above general version of Theorem \ref{thm_largecomponents} guarantee the quasiperiodic behavior of the sequence $(\lind_R \widetilde{I^n})_n$ in the example.
\end{rem}
We do not know any example where the sequence $(\lind_R \widetilde{I^n})_n$ is not quasiperiodic. In view of \cite[Example 4.4]{CHT} on bad behavior of regularity for saturations of powers, it is desirable to seek for one.

\section*{Acknowledgments}
The content of this paper has been worked out when the first author was visiting the Department of Mathematics, University of Nebraska -- Lincoln (UNL) in March 2015. We are grateful to Luchezar Avramov, Dale Cutkosky, Hailong Dao, Thomas Marley and Alexandra Seceleanu for inspiring comments and useful suggestions. The first author would like to thank the colleagues at Department of Mathematics of the UNL, among them Luchezar Avramov, Roger and Sylvia Wiegand for their hospitality and warm support. Finally, both authors are grateful to the anonymous referee for her/his careful reading of the manuscript and many thoughtful comments which greatly improve the presentation.


\begin{thebibliography}{99}
\bibitem{AR}
R. Ahangari Maleki and M.E. Rossi,
\emph{Regularity and linearity defect of modules over local rings}.
J. Commut. Algebra {\bf 6}, no. 4 (2014), 485--504.

\bibitem{Avr1}
L. L. Avramov, 
\emph{Small homomorphisms of local rings}.
J. Algebra {\bf 50} (1978), 400--453.

\bibitem{Br}
M. Brodmann,
\emph{The asymptotic nature of the analytic spread}.
Math. Proc. Cambridge Philos. Soc. {\bf 86} (1979), no. 1, 35--39. 

\bibitem{BH}
W. Bruns and J. Herzog,
\emph{Cohen-Macaulay rings. Rev. ed.}.
Cambridge Studies in Advanced Mathematics {\bf 39}, Cambridge University Press (1998).

\bibitem{Ch}
M. Chardin,
\emph{Powers of ideals and the cohomology of stalks and fibers of morphisms}. 
Algebra Number Theory {\bf 7} (2013), no. 1, 1--18. 

\bibitem{CINR}
A. Conca, S.B. Iyengar, H.D. Nguyen and T. R\"omer,
\emph{Absolutely Koszul algebras and the Backelin-Roos property}.
Acta Math. Vietnam. {\bf 40} (2015), 353--374.

\bibitem{CHT}
S.D. Cutkosky, J. Herzog and N.V. Trung,
{\em Asymptotic behaviour of the Castelnuovo-Mumford regularity}.
 Compositio Math. {\bf 118} (1999), no. 3, 243--261. 
 
\bibitem{Eis}
D. Eisenbud,
\emph{Commutative algebra. With a view toward algebraic geometry}.
Graduate Texts in Mathematics {\bf 150}. Springer-Verlag, New York (1995).
 
\bibitem{EFS}
D. Eisenbud, G. Fl\o ystad and F.~-O. Schreyer,
\emph{Sheaf cohomology and free resolutions over exterior algebras}.
Trans. Amer. Math. Soc. {\bf 355} (2003), 4397--4426. 

\bibitem{GS}
D. Grayson and M. Stillman,
\emph{Macaulay2, a software system for research in algebraic geometry.}
Available at \newblock \verb|http://www.math.uiuc.edu/Macaulay2|.

\bibitem{HaS}
B. Harbourne and A. Seceleanu, 
\emph{Containment counterexamples for ideals of various configurations of
points in $\mathbb P^N$}. 
J. Pure Appl. Algebra {\bf 219} (2015), 1062--1072.

\bibitem{HH}
J. Herzog and T. Hibi,
{\em Componentwise linear ideals}.
Nagoya Math. J. {\bf 153} (1999), 141--153.

\bibitem{HH2}
J. Herzog and T. Hibi,
\emph{The depth of powers of an ideal}.
J. Algebra {\bf 291} (2005), no. 2, 534--550.
 
\bibitem{HHO}
J. Herzog, T. Hibi and H. Ohsugi,
\emph{Powers of componentwise linear ideals}.
Abel Symp. {\bf 6}, Springer, Berlin (2011). 

\bibitem{HIy}
J. Herzog and S.B. Iyengar,
\emph{Koszul modules}. J. Pure Appl. Algebra {\bf 201} (2005), 154--188.

\bibitem{HWY}
J. Herzog, V. Welker and S. Yassemi,
\emph{Homology of powers of ideals: Artin-Rees numbers of syzygies and the Golod property}.
Preprint (2011), \newblock \verb|http://arxiv.org/abs/1108.5862|.

\bibitem{Hu}
C. Huneke,
\emph{The theory of $d$-sequence and powers of ideals}.
Adv. in Math. {\bf 46} (1982), no. 3, 249--279.

\bibitem{HS}
C. Huneke and I. Swanson,
\emph{Integral closure of ideals, rings and modules}.
London Math. Soc., Lecture Note Series, vol. {\bf 336}. Cambridge University Press, Cambridge (2006). 

\bibitem{IyR}
S.B. Iyengar and T. R\"omer,
{\em Linearity defects of modules over commutative rings}.
J. Algebra {\bf 322} (2009), 3212--3237.
 
\bibitem{NS}
U. Nagel and A. Seceleanu,
\emph{Ordinary and symbolic Rees algebras for ideals of Fermat point configurations}.
Preprint (2015), available online at \newblock \verb|http://arxiv.org/abs/1509.04977|.

\bibitem{OkaYan}
R. Okazaki and K. Yanagawa,
\emph{Linearity defect of face rings}.
J. Algebra {\bf 314} (2007), 362--382.
 
\bibitem{Se}
L.M. \c{S}ega,
{\em On the linearity defect of the residue field}.
J. Algebra {\bf 384} (2013), 276--290.

\bibitem{Y}
K. Yanagawa,
\emph{Linearity defect and regularity over a Koszul algebra}.
Math. Scand. {\bf 104} (2009), no. 2, 205--220.
\end{thebibliography}
\end{document}